\documentclass[12pt, a4paper]{amsart}

\makeatletter
\@namedef{subjclassname@2010}{%
  \textup{2010} Mathematics Subject Classification}
\makeatother

\usepackage{hyperref}
 \usepackage{verbatim}
\usepackage{graphicx}
\usepackage{amsmath}            
\usepackage{amsfonts}           
\usepackage{amssymb}            
\usepackage{xcolor}             

\newtheorem{theorem}{Theorem}[section]
 \newtheorem{corollary}[theorem]{Corollary}
 \newtheorem{lemma}[theorem]{Lemma}
 \newtheorem{proposition}[theorem]{Proposition}
 \newtheorem{oprob}[theorem]{Open Problem}
 \theoremstyle{definition}
 
 \theoremstyle{remark}
 \newtheorem{remark}[theorem]{Remark}
 
 \numberwithin{equation}{section}

\newcommand{\om}{\ensuremath{\Omega}}
\newcommand{\cn}{\ensuremath{\mathbb{C}^n}}

\newcommand{\p}{\ensuremath{\partial}}

\DeclareMathOperator{\supp}{supp}

\begin{document}
 \baselineskip=17pt
\title{Supports of Weighted Equilibrium Measures: Complete Characterization}

\author{Muhammed Al\.{i} Alan}
\address[Muhammed Al\.{i} Alan]{Syracuse University, Syracuse, NY,
13244 USA} \email{malan@syr.edu}
\author{N\.{i}hat G\"{o}khan G\"{o}\u{g}\"{u}\c{s}}
\address[N\.{i}hat G\"{o}khan G\"{o}\u{g}\"{u}\c{s}]{ Sabanci University,  Orhanli, Tuzla 34956, Istanbul, TURKEY. E-mail:
nggogus@sabanciuniv.edu}

\thanks{Communicated with Norm Levenberg.}

\date{}

\begin{abstract}
In this paper, we prove that  a compact set $K\subset \cn$ is  the support of a weighted equilibrium measure
if and only  it is  not pluripolar at each of its points extending a result of Saff and Totik to higher dimensions.  Thus, we characterize the supports weighted equilibrium measures completely.  Our proof is a new proof even in one dimension.
\end{abstract}
\maketitle

  \section{Introduction and Background}

The supports of weighted extremal measures $S_w$,  are important in pluripotential theory, approximation theory, complex geometry, and they are  loosely related to parabolic manifolds \cite{AytunaSadullaev}.

Once we know the support of the weighted extremal measure, the weighted extremal function, $V_{K,Q}$, can be determined by solving the homogenous complex Monge-{A}mp\`ere equation in the bounded  components of the complement with boundary value $Q$. Furthermore, $V_{K,Q}=Q$ on the support  $S_w$ quasi everywhere.

Another advantage  of determining  the supports of weighted extremal measures is as follows: The weighted extremal function of $K$  with respect to $Q$ and the weighted extremal function of  the support $S_w$ with respect to the weight $Q|_{S_w}$ are equal. Thus, determining the support of weighted extremal measures makes approximating the weighted capacities very   efficient (see \cite{RajonRansfordRostand}.)

Some applications in weighted approximation are as follows.
By Theorem 2.12 of Appendix B of \cite{Saff-Totik}, a weighted polynomial attains its essential supremum on the support $S_w$. In order to make a  weighted approximation of a continuous function $f$ on $K$, $f$ must vanish outside of $K$. Namely, if $f$ is continuous on $K$ and there is a sequence of weighted polynomials $w^d P_d$ converging uniformly to $f$ on $K$, then $f\equiv0$ on $\subset S_w$ (see \cite{Saff-Totik, Callaghan}.)

Since the weighted extremal function $V_{K,Q}^\ast$ is locally bounded, the weighted extremal measure $(dd^cV_{K,Q}^\ast)^n$ does not put mass on pluripolar sets, i.e.,  $\supp(dd^cV_{K,Q}^\ast)^n$ is   not pluripolar at each of its points; i.e., for all $z \in K$  and all $r > 0, \, B(z,r)\cap K$ is not pluripolar.  It is natural to ask the converse. Namely, if $K$ is a compact set which is not pluripolar at each of its points, then does there exist an admissible weight $Q$ on $K$  such that $\supp(dd^cV_{K,Q}^\ast)^n=K$?

The following theorem gives the converse  in $\mathbb{C}$, which characterizes the supports of weighted extremal measures in $\mathbb{C}$.
\begin{theorem}\cite[Theorem IV.1.1]{Saff-Totik}\label{SSAFFTTTOOTIK} If $K$ is a compact subset of $\mathbb{C}$ which is not pluripolar at each of
its points, then there exists an admissible weight on $K$ such that $\supp (\Delta V_{K,Q})=K$.
\end{theorem}
Unfortunately, the proof of the theorem uses  logarithmic potentials which is not available in \cn.  Branker  and the first author investigated the supports of weighted extremal measures (see \cite{BrankerThesis, AlanMuhammed}. In this paper, we obtain the same theorem \ref{SSAFFTTTOOTIK} in \cn as our main result.

First we recall few facts from weighted and unweighted (pluri-)potential theory. Standard references are \cite{RansfordPotentialTheoryInTheComplexPlane} for unweighted potential theory, \cite{Klimek} for unweighted pluripotential theory, \cite{Saff-Totik} for weighted potential theory, and Appendix B in the same book by Thomas Bloom for weighted pluripotential theory.

Let $K$ be a closed subset of \cn.  An admissible weight function on $K$  is a lower semicontinuous function $Q:K\to (-\infty,\infty]$ such that
\begin{itemize}
\item[i)]  $\{z \in K \, |\, Q(z) <\infty \}$  is not pluripolar.
\item[ii)] If $K$ is unbounded,  then  $Q(z)-\log|z|\to \infty $ as  $|z|\to \infty,\,z\in K$.
\end{itemize}
The function  $w=e^{-Q} $ is also used equivalently in the terminology. Especially, the notation $w$ is used more often in weighted approximation (see \cite{BloomWeightedPolynomialsWeightedPluripotentialTheory, BloomLevenberg, Saff-Totik}.)

The \textbf{weighted  Siciak-Zahariuta extremal function} of $K$ with respect to $Q$  is defined as
\begin{equation}
V_{K,Q}(z):= \sup\left\{ u(z) \mid u\in L, \, u\leq Q \text{ on } K\right\}.
\end{equation}
Recall that $L$ is the Lelong class:
\begin{equation}
 L:=\{ u\, |\, u \text{ is plurisubharmonic on } \cn, u(z) \leq \log^+ |z| +C_u \}.
\end{equation}
If $Q=0$, then $V_{K,0}$  is called the \textbf{(unweighted)  Siciak-Zahariuta extremal function} of $K$  and $V_{K}$ denotes it.

A compact set $K$ is called \textbf{regular} if $V_{K}$ is continuous. If $K\cap \overline{B(z,r)}$ is regular for all $z\in K$ and $r>0$,
the set $K$
is called \textbf{locally regular}. Here we use the notation $B(z_0,r)$ for the open ball of radius $r$ and center $z_0$.

It is well known that the upper semicontinuous regularization of $V_{K,Q}$ is plurisubharmonic and in $L^+$ where
$$L^+ := \{u \in L\, |\, \log^+ |z| +C_u \leq  u(z) \}.$$
Recall that the  \textbf{upper semicontinuous regularization} of a function $v$ is defined by $v^\ast(z):=\limsup\limits_{w\to z}v(w)$.

A subset $P\subset\cn$ is called \textbf{pluripolar} if $E\subset \{z\in \cn \mid
u(z)=-\infty\}$ for some plurisubharmonic function $u$. If a
property holds everywhere except on a pluripolar set we will say that
the property holds \textbf{quasi everywhere}.
It is a well-known fact that $V_{K,Q}=V_{K,Q}^\ast$ quasi everywhere. See \cite{Klimek}.

 Let $S_w$ denotes the support of  the $(dd^cV_{K,Q}^\ast)^n$, where $(dd^c u)^n $ is the  Monge-{A}mp\`ere measure of $u$.  The following lemma is very useful to determine the supports of Monge-{A}mp\`ere measures.
\begin{lemma}\label{SuppprtOfExtremalMeasureSubsetSw} \cite[Appendix B, Theorem 1.3]{Saff-Totik}  Let $S_w^\ast:= \{z\in \cn \mid V_{K,Q}^\ast(z)\geq Q(z) \}$.
Then we have $S_w\subset S_w^\ast$.
\end{lemma}

\begin{theorem}\label{DemaillyMaximumPrinciple}\cite[Proposition 11.9]{DemaillyPotentialTheoryinSeveralComplexVariables}
Let $u, v$  be locally bounded plurisubharmonic functions on \om. Then we have the following inequality
\begin{equation}\label{DemaillyInequality}
(dd^c\max\{u,v\})^n\geq \chi_{\{u\geq v\}}(dd^c u)^n+ \chi_{\{u<  v\}}(dd^c v)^n .
\end{equation}
\end{theorem}
Here $\chi_A $ is the \textbf{characteristic function} of $A$. The inequality \eqref{DemaillyInequality} will be called the
\textbf{Demailly inequality}.

\begin{proposition} \cite[Proposition 2.13]{SiciakExtremal1} \label{Siciak213}
If $K$ is locally regular and $Q$ is continuous, then $V_{K,Q}$ is continuous.
\end{proposition}

 \section{Characterization of the Supports}

\begin{proposition}\label{proposition12}
Let $K$ be a non-pluripolar compact set in $\cn$ and let $u $ be  a continuous  plurisubharmonic function in Lelong class.
If $Q$ is the weight on $K$ defined by $Q := u|_K$, then  we have $V_{K,Q} = u$ on $K$.
\end{proposition}
\begin{proof}
Because $u$ itself is a competitor in the envelope defining $V_{K,Q }$, we have  $u\leq V_{K,Q }$ on
\cn ; and $ V_{K,Q } \leq   Q = u $ on $K$. Thus $V_{K,Q} = u$ on $K$.
\end{proof}

Note that $u = V_{K,Q} ^\ast$ quasi everywhere on $K$; i.e., we have $u = V_{K,Q} ^\ast$ on $K \setminus P$ where $P$ is a pluripolar set.
The following theorem is our main result which gives the complete characterization of supports of weighted extremal measures.

\begin{theorem}
Let $K$ be a  compact set in $\cn$ which is not pluripolar at each of
its points; i.e., for all $z \in K$  and all $r > 0, \, B(z,r)\cap K$ is not pluripolar. There
exists a continuous weight $Q$ on $K$ so that $K =\supp(dd^c V_{K,Q}^\ast)^n$.
\end{theorem}
\begin{proof} Since $K$ is compact, $K\subset K_r$ for some $r>0$, where
 $K_r:=\overline{B(z,r)}$. Let   $Q_r$ be the weight on $K_r$ defined by $Q_r:=\frac{1}{\sqrt{2r}}|z|^2$. By Example 3.7 of \cite{AlanMuhammed}, we have  $\supp (dd^c V_{K_r, Q_r})^n = K_r$.

We define  $Q|_K:=u=V_{K_r, Q_r}$. By proposition \ref{proposition12} we have $V_{K,Q}^\ast =u$ quasi everywhere on $K$.
By Demailly's inequality we have
  \begin{eqnarray}
    \nonumber  (dd^c V_{K, Q}^\ast)^n&=& (dd^c \max\{ V_{K, Q}^\ast, V_{K_r, Q_r} )^n  \\
    \nonumber                               &\geq &\chi_{\{ V_{K_r, Q_r}\geq V_{K,Q}^\ast\}} (dd^c V_{K_r, Q_r})^n +\chi_{ \{V_{K, Q}^\ast > V_{K_r, Q_r}\}} (dd^c V_{K, Q}^\ast)^n.
  \end{eqnarray}
Due to the facts that the set $\{V_{K, Q}^\ast > V_{K_r, Q_r}\}\cap K$ is pluripolar,  and that $ V_{K, Q}^\ast$ is  locally bounded, we have $(dd^c V_{K, Q}^\ast)^n $ vanishes on $\{V_{K, Q}^\ast > V_{K_r, Q_r}\}\cap K$. Therefore, we have
$(dd^c V_{K, Q}^\ast)^n \geq (dd^c V_{K_r, Q_r})^n$ quasi everywhere on $K$. Namely,  for any non-pluripolar (Borel) subset $E$   of $K$, we have
\begin{equation}\label{EQ12}(dd^c V_{K, Q}^\ast)^n (E)\geq (dd^c V_{K_r, Q_r})^n(E)>0.\end{equation}
For any $z\in K$ for every $r>0, $ we have $(dd^c V_{K, Q}^\ast)^n (K\cap B(z,r) )>0$. Therefore $z\in \supp(dd^c V_{K, Q}^\ast)^n$.
\end{proof}

\begin{corollary}\label{corallary101}
Let $K$ be a locally regular compact subset of \cn. Then there exists a continuous weight $Q$ on $K$ such that $K =\supp(dd^c V_{K,Q}^\ast)^n$ and $Q=V_{K,Q}$ on $K$.
\end{corollary}

\begin{proof}
We define $K_r$ and $Q_r$ as in the proof of above theorem. By above theorem we have
$K =\supp(dd^c V_{K,Q}^\ast)^n$.
By Proposition \ref{Siciak213}, we have $V_{K,Q}$ is continuous, thus $V_{K,Q}\leq Q$ on $K$. By combining these with Lemma \ref{SuppprtOfExtremalMeasureSubsetSw}, we have $Q=V_{K,Q}$ on $K$.
\end{proof}

As a corollary, we obtain the following unexpected result.
\begin{corollary}
There exists  a continuous plurisubharmonic function $u\in L^+$, such that $\supp(dd^c u)^n =\p \Delta^n$, where $\Delta^n$ is the polydisc in \cn.
\end{corollary}

\begin{oprob}A compact set $K\subset\cn$ is locally regular if and only  it is the support of the  Monge-{A}mp\`ere measure of a continuous function in $L^+$.
\end{oprob}

\begin{remark}
Note that the above open problem might be a step to understand the measures which are  Monge-{A}mp\`ere measures of  continuous plurisubharmonic function.
\end{remark}

\bibliographystyle{amsalpha}

\providecommand{\bysame}{\leavevmode\hbox to3em{\hrulefill}\thinspace}
\providecommand{\MR}{\relax\ifhmode\unskip\space\fi MR }
\providecommand{\MRhref}[2]{%
  \href{http://www.ams.org/mathscinet-getitem?mr=#1}{#2}
}
\providecommand{\href}[2]{#2}

\end{document}